\RequirePackage{fix-cm}
\documentclass{article}
\usepackage[utf8]{inputenc}
\usepackage{tikz-cd}
\usetikzlibrary{cd}
\usepackage{amsmath,amssymb,amsthm,amsfonts,amscd, url,mathtools}
\usepackage{mathrsfs}
\usepackage{tikz}
\usetikzlibrary{calc}
\usepackage{zref-savepos}
\usepackage{tikz}
\usepackage{graphicx,wrapfig}
\usepackage{float}
\graphicspath{ {images/} } 
\usepackage[english]{babel}
\usepackage[autostyle]{csquotes}
\usepackage{hyperref}
\usepackage{geometry}

\DeclareMathSizes{10}{10}{6}{4}

\DeclareFontFamily{U}{MnSymbolC}{}
\DeclareSymbolFont{MnSyC}{U}{MnSymbolC}{m}{n}
\DeclareFontShape{U}{MnSymbolC}{m}{n}{
    <-6>  MnSymbolC5
   <6-7>  MnSymbolC6
   <7-8>  MnSymbolC7
   <8-9>  MnSymbolC8
   <9-10> MnSymbolC9
  <10-12> MnSymbolC10
  <12->   MnSymbolC12}{}
\DeclareMathSymbol{\intprod}{\mathbin}{MnSyC}{'270}

\newtheoremstyle{problemstyle}
{\topsep} % Space above
{20pt} % Space below
{} % Body font
{} % Indent amount
{\scshape} % Theorem head font
{\newline} % Punctuatnion after theorem head
{1em} % Space after theorem head
{} % Theorem head spec (can be left empty, meaning `normal')

\newtheorem{defs}{Definition}
\numberwithin{defs}{section}
\newtheorem{prop}{Proposition}
\numberwithin{prop}{section}
\newtheorem{egs}{Example}
\newtheorem{cor}{Corollary}
\numberwithin{cor}{section}
\newtheorem{lem}{Lemma}
\numberwithin{lem}{section}
\newtheorem{mthm}{Main Theorem}
\newtheorem{thm}{Theorem}
\numberwithin{thm}{section}
\theoremstyle{problemstyle}
\newtheorem{rem}{Remark}

\newcommand{\spn}{\text{span}}

\newcommand{\into}{\hookrightarrow}

\newcommand{\bR}{\mathbb{R}}
\newcommand{\bC}{\mathbb{C}}

\newcommand{\msS}{\mathscr{S}}

\newcommand{\mfX}{\mathfrak{X}}
\newcommand{\mft}{\mathfrak{t}}
\newcommand{\mfF}{\mathfrak{F}}

\newcommand{\Hom}{\text{Hom}}

\newcommand{\bge}{\begin{equation*}}
\newcommand{\ene}{\end{equation*}}

\newcommand{\cshf}{\underline{\bC}_M^\times}

\title{Invariance of Polarization Induced by Symplectomorphisms}
\author{Ethan Ross }
\date{2021}

\begin{document}

\maketitle

\section{Introduction}

Geometric quantization roughly amounts to associating complex vector spaces (preferably Hilbert spaces) to symplectic manifolds. The terminology arises from the realization of a symplectic manifold as a classical phase space and a Hilbert space as the space of quantum wave functions. One fruitful direction in quantization has been the Kirillov-Kostant-Souriau picture in which one associates to a given symplectic manifold $(M,\omega)$ a complex Hermitian line bundle with compatible covariant derivative $(L,\nabla)\to M$ so that the curvature $curv(\nabla)$ is given by the symplectic form $\omega$. The Hilbert space of $L^2$ sections then satisfies many of the naive axioms of quantization given by Dirac \cite[axioms Q1-Q3, page 155]{wood}, but it can be \enquote{too large} in some sense. In order to correct this issue with size, one possible approach is to introduce an object called a polarization. In this paper I shall only be considering real polarizations, which are possibly singular Lagrangian foliations $P$ of the symplectic manifold $(M,\omega)$. K\"{a}hler polarizations, given by complex structures compatible with the symplectic form $\omega$, are also widely used in the literature, but shall not be considered in this paper.

From here, one has at least two choices for what a quantization could be. One option uses the covariant derivative $\nabla$, to define a sheaf of polarized sections $\msS_{(P,\nabla)}$ of $L$. I then define the sheaf quantization to be
\begin{equation}\label{shfquant}
Q_{shf}(M,P,\nabla)=\bigoplus_{n}\check{H}^n(M,\msS_{(P,\nabla)}),
\end{equation}
where $\check{H}^n(M,\msS_{(P,\nabla)})$ are the \v{C}ech cohomology groups associated to the sheaf. The other option makes use of the fact that $P$ is a foliation to define distinguished leaves $\iota:B\to M$, called Bohr-Sommerfeld leaves,  which admit non-trivial covariantly constant sections, that is, sections $s$ of the pull back bundle $(\iota^*L,\iota^*\nabla)$ satisfying $\iota^*\nabla s=0$ along $TB$. Writing $BS(P)$ for the Bohr-Sommerfeld leaves of $P$, I then define the other quantization, Bohr-Sommerfeld quantization, by
\begin{equation}\label{bohrquant}
Q_{BS}(M,P,\nabla)=\bigoplus_{B\in BS(P)} \check{H}^0(B,\msS_{(P,\nabla)}|_B),
\end{equation}
where $\msS_{(P,\nabla)}|_B$ denotes the sheaf of covariantly constant sections on $B$. Sometimes these two quantizations agree as shown by Sniatycki \cite[Theorems 1.1 and 1.2]{coho}, and other times they do not as shown by Hamilton \cite[Theorem 8.10]{ham}. Since both are of interest, I will consider both in this paper.

A famous body of results in geometric quantization are the \enquote{invariance of polarization} results, where it can be shown that two naturally arising polarizations induce the same quantization. Standard examples include the Gelfand-Zeitlin system \cite[Theorem 6.1]{guil} and the moduli space of flat $SU(2)$ connections \cite[Theorem 8.3]{lisa}. In this paper, I show a new kind of invariance theorem arising from the action of a symplectomorphism on the polarizations. The reason this invariance is \enquote{new} is due to the fact that I am comparing quantizations coming from two real polarizations. The more classic results cited above arose from comparing the quantizations of a real polarization with a K\"{a}hler polarization. 

In particular, given a real polarization $P$ and a symplectomorphism $\Phi$ on $(M,\omega)$, I define a new real polarization $\Phi^*P$ via the pushforward of the inverse of $\Phi$. It then appears to be a tautology that the quantizations of $M$ with respect to $P$ and $\Phi^*P$ should agree, however I could only show this holds if the symplectomorphism $\Phi$ lifts to a connection-preserving isomorphism on the prequantum line bundle $(L,\nabla)\to M$. It is not known to the author if this condition is necessary.

To set up the statements of the main results, let $\cshf$ denote the sheaf of locally constant, non-vanishing complex functions on $M$. I then obtain the following two theorems.

\begin{mthm}
Let $(M,\omega)$ be a symplectic manifold, $P$ a real polarization, $\Phi:M\to M$ a symplectomorphism, and $(L,\nabla)\to M$ a prequantum line bundle. If $\check{H}^1(M,\cshf)=0$, then $\Phi$ induces an isomorphism between the sheaf quantizations
\bge
Q_{shf}(M,P,\nabla)\to Q_{shf}(M,\Phi^*P,\nabla),
\ene
where $Q_{shf}(M,P,\nabla)$ is defined in Equation \ref{shfquant}.
\end{mthm}

\begin{mthm}
Let $(M,\omega)$ be a symplectic manifold, $P$ a real polarization, $\Phi:M\to M$ a symplectomorphism, and $(L,\nabla)\to M$ a prequantum line bundle. If $\check{H}^1(M,\cshf)=0$, then $\Phi$ induces an isomorphism between the Bohr-Sommerfeld quantizations
\bge
Q_{BS}(M,P,\nabla)\to Q_{BS}(M,\Phi^*P,\nabla),
\ene
where $Q_{BS}(M,P,\nabla)$ is defined in Equation \ref{bohrquant}
\end{mthm}

The upshot of the proofs of these theorems is that the isomorphism can explicitly be defined via lifts of the symplectomorphism $\Phi$ to the prequantum line bundle. Historically, most invariance results arise from counting the dimensions of the quantizations with respect to two choices of polarization, then showing these agree. So, even though the isomorphisms of Main Theorems 1 and 2 depend on choices, they could be considered more \enquote{canonical} in some sense. 

Once again, it should be noted here that although Main Theorems 1 and 2 are invariance of polarization results, they are of a different nature than the classical results cited above. In particular, I am only comparing the quantizations arising from two real polarizations. There is a more general notion of polarization where one consider subbundles $P$ of the complexified tangent bundle $T_\bC M$. This gives rise to three classes of polarizations: real (which I am considering in this paper), mixed, and complex. In the cases of the Gelfand-Zeitlin system and the moduli space of flat $SU(2)$ connections, the main point of interest is that quantization by naturally arising real and complex polarizations give the same quantization. The ideas in this paper do not apply to those results since symplectomorphisms do not change the \enquote{type} of a polarization. An interesting future direction of research could involve finding larger classes of symmetries than just symplectomorphisms which interpolate between the various kinds of polarizations and investigating if similar invariance results can be obtained. I am also unaware of any example of a symplectic manifold $(M,\omega)$ with $\check{H}^1(M,\cshf)\neq 0$ such that quantization is not preserved under the action of symplectomorphisms. This could also be an interesting area of future research. 

I would like to thank Dr. Lisa Jeffrey, Dan Hudson, Carrie Clark, and the referee for reading through the drafts of this paper and pointing out the numerous grammatical mistakes. 

\section{Quantization}
Let $(M,\omega)$ be a symplectic manifold.

\begin{defs}
A prequantum line bundle over $M$ is a complex Hermitian line bundle with compatible connection $(L,\nabla)\to M$ such that 
\bge
curv(\nabla)=\omega.
\ene
\end{defs}

Naively, a quantization should attach to a symplectic manifold (thought of as a classical phase space) a Hilbert space and a map taking the classical observables $C^\infty(M)$ to self-adjoint operators on the Hilbert space. Prequantum line bundles achieve both of these objectives. The Hermitian structure of the complex line bundle together with the canonical volume form $\omega^n/n!$ allows us to equip the space $\Gamma_c(L)$ of compactly supported sections with the structure of a pre-Hilbert space. Then, using the compatible covariant derivative $\nabla$, one can construct the desired map between classical and quantum observables. See Woodhouse \cite[Chapter 8]{wood} for more details. 

The reason why these are called \enquote{prequantum} line bundles and not \enquote{quantum} line bundles is due to the fact that even in the simplest cases, the resulting Hilbert space is too large in some sense. For example, equip $\bR^{2n}$ with standard coordinates $(x_1,\dots,x_n,y_1,\dots,y_n)$ and the standard symplectic structure 
\bge
\omega_0=\sum_j dx_j\wedge dy_j.
\ene 
Also define the prequantum line bundle $(L,\nabla)$ over $(\bR^{2n},\omega_0)$ by
\bge
L=\bR^{2n}\times \bC,\quad \nabla=d+i\sum_{j=1}^nx_jdy_j,
\ene
where I identify the sections of $L$ with smooth complex-valued functions on $\bR^{2n}$. It's then an easy exercise to show that the resulting Hilbert space is $L^2(\bR^{2n})$. This is unsatisfactory since  the quantization of symplectic $\bR^{2n}$ should be $L^2(\bR^n)$. Furthermore, important properties like the Heisenberg uncertainty relations between the coordinates will not hold in this Hilbert space. 

An easy remedy for these problems would be to only consider functions $f$ which only depend on half the variables, say $f=f(y_1,\dots,y_n)$. Observe that these are precisely the functions which obey
\bge
\nabla_{\frac{\partial}{\partial x_j}}f=0
\ene
for each $j$, that is, functions which are covariantly constant along the Lagrangian subbundle 
\bge
P=\text{span}\{\frac{\partial}{\partial x_1},\dots,\frac{\partial}{\partial x_n}\}.
\ene
Generalizing this idea is where the concept of a polarization arises. 

\begin{defs}\label{poldef}
A (singular, real) polarization of $(M,\omega)$ is a singular subbundle $P\subset TM$ such that its sheaf of sections $\mfF_P$ satisfies the following axioms.
\begin{itemize}
    \item[(i)] (Involutivity) If $X,Y\in \mfF_P$, then so is $[X,Y]$.
    \item[(ii)] (Locally Finitely Generated) For any $x\in M$, there exists an open neighbourhood $U\subset M$ of $x$ and sections $X_1,\dots,X_k\in \mfF_P(U)$ such that
    \bge
    \mfF_P(U)=\spn_{C^\infty(U)}\{X_1,\dots,X_k\}.
    \ene
    \item[(iii)] (Lagrangian) There exists open dense subset $U\subset M$ such that for each $x\in U$, $P_x\subset T_xM$ is Lagrangian.
\end{itemize}
\end{defs}

\begin{rem}
\begin{itemize}
    \item[(i)] Involutivity and the Locally Finitely Generated property of $P$ means that it defines a singular foliation in the sense of \cite{holo}. It then automatically follows that there exists an open dense subset $U\subset M$ such that $P|_U$ is an involutive smooth subbundle of $TU$. 
    \item[(ii)]As was noted in the introduction, there are other kinds of polarizations. In the non-singular case, a polarization is an involutive Lagrangian subbundle $P\subset T_\bC M$, where $T_\bC M$ is the complexified tangent bundle. Usually it is also demanded that $E=(P+\overline{P})\cap TM$ and $D=P\cap \overline{P}\cap TM$ are also subbundles. $P$ is called real if $D=E$, complex if $E=TM$ and $D=0$, and mixed otherwise. See Andersen \cite{jorgen} for a more in depth discussion.
\end{itemize}
\end{rem}

Since Lagrangian subspaces have half the dimension of the symplectic manifold, demanding sections be covariantly constant along the directions associated to $P$ is one mechanism for cutting down the variables on which the Hilbert space depends. With this motivation, define the sheaf of covariantly constant sections $\msS_{(P,\nabla)}$ by
\begin{equation}
\msS_{(P,\nabla)}(U):=\{s\in\Gamma(U,L) \ | \ \nabla_X s=0 \ \forall \ X\in \mfF_P(U)\},
\end{equation}
for open $U\subset M$.

This enables us to formally define the sheaf quantization for the purposes of this paper.

\begin{defs}
Define the sheaf quantization of $(M,\omega)$ with respect to the prequantum line bundle $(L,\nabla)\to M$ and the polarization $P\subset TM$ by
\begin{equation}
Q_{shf}(M,P,\nabla):=\bigoplus_{n}\check{H}^n(M,\msS_{(P,\nabla)}),
\end{equation}
where $\check{H}^n(M,\msS_{(P,\nabla)})$ denotes the $n$-th \^{C}ech cohomology group of $M$ with respect to the sheaf $\msS_{(P,\nabla)}$.
\end{defs}

As alluded to in the introduction, another common way of quantizing a symplectic manifold is by a Bohr-Sommerfeld quantization. To do this, observe that if $P\subset TM$ is a polarization, then $P$ induces a decomposition of $M$ into disjoint immersed submanifolds called leaves \cite[Theorem 4.2]{suss}. If $M$ is connected, the leaves of maximal dimension are immersed Lagrangian submanifolds.

Now, if $\iota:B\to M$ is a leaf of $P$, one can pull back the covariant derivative $\nabla$ on $L$ to the pull back bundle $\iota^*L\to B$ and write $\iota^*\nabla$ for the resulting covariant derivative. Thus, define a sheaf $\msS_{(P,\nabla)}|_B$ on $B$ by
\begin{equation}
\msS_{(P,\nabla)}|_B(U)=\{s\in\Gamma(U,\iota^*L) \ | \ \iota^*\nabla_Xs=0 \ \forall X\in\mfX(U)\},
\end{equation}
for open $U\subset B$. 

\begin{defs}
A leaf $B$ of $P$ is called Bohr-Sommerfeld if
\bge
H^0(B,\msS_{(P,\nabla)}|_B)\neq 0.
\ene
Writing $BS(P)$ for the set of Bohr-Sommerfeld leaves of $P$, define the Bohr-Sommerfeld quantization of $(M,\omega)$ with respect to the prequantum line bundle $(L,\nabla)\to M$ and the polarization $P$ by
\begin{equation}
Q_{BS}(M,P,\nabla):=\bigoplus_{B\in BS(P)} H^0(B,\msS_{(P,\nabla)}|_B).
\end{equation}
\end{defs}

\begin{rem}
\begin{itemize}
    \item[(i)] An equivalent definition of a Bohr-Sommerfeld leaf used in the literature is as follows. A Bohr-Sommerfeld leaf is a leaf $\iota:B\to M$ that admits a non-trivial covariantly constant section $s:B\to \iota^*L$, i.e. a covariantly constant section which is not identically the zero section.
    \item[(ii)]If $B$ is a Bohr-Sommerfeld leaf and is connected, then it follows that 
    \bge
    H^0(B,\msS_{(P,\nabla)}|_B)\cong \bC.
    \ene
\end{itemize}
\end{rem}

\section{Lifts of Symplectomorphisms}
The main mechanism for comparing the quantizations by various polarizations is by lifts of symplectomorphisms to automorphisms of the line bundle. For the following discussion, fix a prequantum line bundle $(L,\nabla)\xrightarrow{\pi} (M,\omega)$, with $M$ connected, and a symplectomorphism $\Phi:M\to M$.

\begin{defs}
A lift of the symplectomorphism $\Phi$ is a diffeomorphism $F:L\to L$ such that
\begin{itemize}
    \item[(i)] $\pi\circ F=\Phi\circ \pi$.
    \item[(ii)] For each $x\in M$, the map
    \bge
    F_x:L_x\to L_{\Phi(x)}
    \ene
    is a linear isomorphism.
\end{itemize}
\end{defs}

\begin{egs}
If $L=M\times \bC$ is a trivial line bundle, then a lift of $\Phi:M\to M$ is given by
\bge
F:L\to L;\quad (x,z)\mapsto (\Phi(x),f(x)z),
\ene
where $f\in C^\infty(M,\bC^\times)$ is a non-vanishing smooth function. 
\end{egs}

\begin{rem}
Let $X$ and $Y$ be topological spaces and $f:X\to Y$ a continuous map. $f$ then defines a functor between the category of sheaves on $X$ to the category of sheaves on $Y$. Indeed, let $\msS$ be a sheaf of $X$ and define sheaf $f_*\msS$ on $Y$ by
\begin{equation}\label{pushsheaf}
f_*\msS(U):=\msS(f^{-1}(U))
\end{equation}
for open $U\subset Y$. Similarly, given a natural transformation
\bge
\eta:\msS\to \msS'
\ene
between sheaves $\msS$ and $\msS'$ on $X$, we can define natural transformation
\bge
f_*\eta:f_*\msS\to f_*\msS'
\ene
by
\begin{equation}\label{pushnat}
f_*\eta(U)=\eta(f^{-1}(U)):\msS(f^{-1}(U))\to \msS'(f^{-1}(U))
\end{equation}
for open $U\subset Y$. 
\end{rem}

Write $\Gamma_L$ for the sheaf of sections of $L$. A lift $F$ of $\Phi$ then defines a natural isomorphism
\begin{equation}\label{nattran}
F_*:\Phi_*\Gamma_L\to \Gamma_L,
\end{equation}
where for each open $U\subset M$ and each $s\in\Gamma_L(\Phi^{-1}(U))$, define $F_*s\in\Gamma_L(U)$ by
\bge
F_*s(x):=F_{\Phi^{-1}(x)}(s(\Phi^{-1}(x))),\quad x\in U.
\ene
It's easy to see that $F_*$ respects restrictions and hence is indeed a natural transformation. The inverse to $F$ thus also induces a natural transformation
\bge
(F^{-1})_*:(\Phi^{-1})_*\Gamma_L\to \Gamma_L.
\ene
Pushing this forward by $\Phi$, 
\bge
\Phi_*(F^{-1})_*:\Gamma_L\to \Phi_*\Gamma_L,
\ene
we thus get our inverse to $F_*$.

Now, let $\mfX_M$ denote the sheaf of vector fields on $M$. Since $\Phi$ is a diffeomorphism, its derivative $d\Phi$ also defines a natural isomorphism
\bge
\Phi_*:\Phi_*\mfX_M\to \mfX_M.
\ene

Thus, since the covariant derivative $\nabla$ is a natural transformation $\mfX_M\otimes \Gamma_L\to \Gamma_L$, the natural isomorphisms $\Phi_*$ and $F_*$ are then used to define a pullback covariant derivative.

\begin{defs}
Given a lift $F:L\to L$ of $\Phi$, define the pullback covariant derivative $F^*\nabla$ to be the unique natural transformation
\bge
F^*\nabla:\mfX_M\otimes \Gamma_L\to \Gamma_L
\ene
so that the diagram commutes
\bge
\begin{tikzcd}
\Phi_*(\mfX_M\otimes \Gamma_L)\arrow[rr,dashed,"\Phi_*(F^*\nabla)"]\arrow[dd,"\Phi_*\times F_*"] && \Phi_*\Gamma_L\arrow[dd,"F_*"]\\
&&\\
\mfX_M\otimes \Gamma_L\arrow[rr,"\nabla"] && \Gamma_L
\end{tikzcd}
\ene
\end{defs}

\begin{rem}
For any lift $F$ of $\Phi$, $F^*\nabla$ is a prequantum covariant derivative for $L\to (M,\omega)$, that is, the curvature of $F^\nabla$ is the symplectic form $\omega$,
\bge
\text{curv}(F^*\nabla)=\omega.
\ene
This can be shown as follows. Let $U\subset M$ be open and $s\in\Gamma_L(U)$ non-vanishing, then there exists unique $\alpha\in\Omega^1(U)$ such that
\begin{equation}
\nabla s=i\alpha\otimes s.
\end{equation}
Since $\text{curv}(\nabla)=\omega$, it follows that $d\alpha=\omega|_{U}$. Now, define $t\in\Gamma_L(\Phi^{-1}(U))$ uniquely by $F_*t=s$. It's then a matter of unfolding the above definitions to see that $t$ is non-vanishing and that
\begin{equation}
F^*\nabla t=i\Phi^*\alpha\otimes t.
\end{equation}
Thus, since $\Phi$ is a symplectomorphism, 
\bge
d\Phi^*\alpha=\omega|_{\Phi^{-1}(U)}.
\ene
Hence, $\text{curv}(F^*\nabla)=\omega$.
\end{rem}

It will be important later to determine when a symplectomorphism $\Phi$ admits a covariant derivative preserving lift, that is, a lift $F:L\to L$ satisfying
\bge
F^*\nabla=\nabla.
\ene
If $\Phi$ does admit such a lift, then $\Phi$ will always preserve the quantization with respect to any polarization.

\begin{lem}\label{KEYLEM}
Let $\cshf$ be the sheaf of locally constant $\bC^\times$-valued functions on $M$. If $\check{H}^1(M,\cshf)=0$, then $\Phi$ admits a lift $F:L\to L$ satisfying $F^*\nabla=\nabla$.
\end{lem}

\begin{proof}
Choose a good cover $\{U_j\}$\cite{bott}, that is, for any finite collection of indices $j_1,\dots,j_k$, the intersection 
\begin{equation}
U_{j_1\cdots j_k}:=U_{j_1}\cap\cdots\cap U_{j_k}
\end{equation}
is either empty or contractible. One can always choose such a cover since contractible covers are cofinal among all covers. Furthermore, since $U_j$ is contractible for each $j$, there exists a non-vanishing section $s_j\in\Gamma_L(U_j)$. 

Since each of the $s_j$ are non-vanishing, we obtain useful local data.
\begin{itemize}
    \item (Local primitives of $\omega$) As was noted above, on each $U_j$ there exists unique $\alpha_j\in\Omega^1(U_j)$ such that
    \begin{equation}
    \nabla s_j=i\alpha_j\otimes s_j.
    \end{equation}
    These $\alpha_j$ all satisfy
    \begin{equation}
    d\alpha_j=\omega|_{U_j}.
    \end{equation}
    
    \item (Transition Functions) If $U_{jk}\neq \emptyset$, then there exists unique $\lambda_{jk}\in C^\infty(U_{jk},\bC^\times)$ satisfying
    \begin{equation}
    \lambda_{jk} s_k|_{U_{jk}}=s_j|_{U_{jk}}.
    \end{equation}
\end{itemize}
It's a straightforward calculation to show that on overlaps $U_{jk}$, the local primitives $\alpha_j$ and the transition functions $\lambda_{jk}$ are related by
\begin{equation}\label{tran2prim}
\alpha_j-\alpha_k=-id\log(\lambda_{jk})
\end{equation}
for any choice of branch of $\log$.

Now, using the inverse image of $\Phi$, we obtain another good cover $\{\Phi^{-1}(U_j)\}$, and so we can choose a new collection of non-vanishing local sections where $t_j\in\Gamma_L(\Phi^{-1}(U_j))$. Write  $\beta_j$ for the associated local primitives of $\omega$ and write $\mu_{jk}$ for the associated transition functions. The sections $t_j$ will now be suitably re-scaled to define the lift $F$.

First, we may assume that $\beta_j=\Phi^*\alpha_j$. Otherwise, since $U_j$ is contractible and since $\beta_j$ and $\Phi^*\alpha_j$ are local primitives of $\omega$, there exists $f_j\in C^\infty(\Phi^{-1}(U_j))$ such that
\bge
\beta_j-\Phi^*\alpha_j=df_j.
\ene
Now redefine $t_j':=e^{-if_j}t_j$. It then follows that
\bge
\nabla t_j'=i\Phi^*\alpha_j\otimes t_j'.
\ene

Next, we may assume that the transition functions $\mu_{jk}$ for the $t_j$ are related to the transition functions of the $s_j$ by
\bge
\mu_{jk}=\lambda_{jk}\circ \Phi.
\ene
Indeed, otherwise using equation (\ref{tran2prim}), we have
\bge
-id\log(\lambda_{jk})=\alpha_j-\alpha_k
\ene
and, by assumption, we also have
\bge
-id\log(\mu_{jk})=\Phi^*\alpha_j-\Phi^*\alpha_k.
\ene
Thus,
\bge
d\log(\mu_{jk})=d\log(\Phi^*\lambda_{jk}).
\ene
Thus, since $U_{jk}$ is contractible, there exists $C_{jk}\in \bC^\times$ such that
\bge
\mu_{jk}=C_{jk}\Phi^*\lambda_{jk}.
\ene
Observe that on triple overlaps $U_{jk\ell}$ the following cocycle condition holds
\bge
\lambda_{k\ell}\lambda^{-1}_{j\ell}\lambda_{jk}=1.
\ene
Thus, the constants $C_{jk}$ satisfy an analogous cocycle condition
\bge
C_{k\ell}C^{-1}_{j\ell}C_{jk}=1.
\ene
Hence, the constants $C_{jk}$ define a closed 2-cocycle $\{C_{jk}\}\in \check{Z}^1(\{U_j\},\cshf)$. Since  $\check{H}^1(\{U_j\},\cshf)=0$, there exists a collection of constants $\{e_j\}\subset \bC^\times$ such that
\bge
C_{jk}=e_ke_j^{-1}.
\ene
Now, redefine $t'_j:=e_jt_j$. It's then straightforward to check that
\bge
\Phi^*\lambda_{jk} t_k'=t_j'
\ene
and the local primitives are still $\Phi^*\alpha_j$
\bge
\nabla t_j'=i\Phi^*\alpha_j\otimes t_j'.
\ene

Thus,  we may locally define lifts $F_j$ of $\Phi$. For each $j$, set
\bge
F_j:L|_{\Phi^{-1}(U_j)}\to L|_{U_j}
\ene
uniquely by $(F_j)_*t_j=s_j$. Since the transition functions for $t_j$ are given by the pullbacks of the transition functions of the $s_j$, it follows that $F_j=F_k$ on overlaps $U_{jk}$. Thus, we obtain a global lift $F:L\to L$ of $\Phi$. By construction, we have
\bge
\nabla t_j=i\Phi^*\alpha_j\otimes t_j.
\ene
Further,
\bge
F^*\nabla t_j=i\Phi^*\alpha_j\otimes t_j.
\ene
Therefore, $F^*\nabla=\nabla$.

\end{proof}

\section{Action by Symplectomorphisms on Polarizations}
Consider the motivating example for polarizations, $\bR^{2n}$ with the standard prequantum line bundle introduced in the beginning of section 2. Physicists will (implicitly) use one of two polarizations to cut down on the number of variables. Recall,  $\bR^{2n}$ is given coordinates $x_1,\dots,x_n,y_1,\dots,y_n$, then there are two naturally arising polarizations:
\bge
P=\text{span}\{\frac{\partial}{\partial x_1},\dots,\frac{\partial}{\partial x_n}\}
\ene
and
\bge
Q=\text{span}\{\frac{\partial}{\partial y_1},\dots,\frac{\partial}{\partial y_n}\}.
\ene
Using $P$ as the polarization gives the \enquote{momentum representation} $Q_{shf}(\bR^{2n},P)$ and $Q$ returns the \enquote{position representation} $Q_{shf}(\bR^{2n},Q)$. If one were to equip these vector spaces with a Hilbert space structure, this can be done using half-forms\cite{jorgen}, then both would be isomorphic to $L^2(\bR^n)$. In particular, they could be viewed as the same quantization. 

Another, perhaps more geometric approach would be to realize that 
\bge
\Phi:\bR^{2n}\to \bR^{2n};\quad (x_1,\dots,x_n,y_1,\dots,y_n)\mapsto (y_1,\dots,y_n,-x_1,\dots,-x_n)
\ene
is a symplectomorphism which fibrewise swaps the polarizations
\bge
\Phi_*P=Q,\quad \Phi_*Q=P.
\ene
Furthermore, if $U\subset \bR^{2n}$ is open, then it's easy to see that precomposition by $\Phi$ defines an isomorphism of sheaves
\bge
\Phi^*:\msS_{(P,\nabla)}\to \Phi_*\msS_{(Q,\nabla)}.
\ene
Hence, $\Phi$ induces an isomorphism 
\bge
Q_{shf}(\bR^{2n},P)\to Q_{shf}(\bR^{2n},Q)
\ene
between the momentum and position representations as desired.

I now want to generalize the above construction to more general prequantum line bundles and polarizations. 

\begin{defs}
Fix a symplectic manifold $(M,\omega)$, a symplectomorphism $\Phi:M\to M$, and a polarization $P\subset TM$. Define a new polarization $\Phi^*P\subset TM$ point-wise by
\begin{equation}\label{polonsymp}
\Phi^*P_x:=d_{\Phi(x)}\Phi^{-1}(P_{\Phi(x)}).
\end{equation}
\end{defs}

Since $\Phi$ is a diffeomorphism it easily follows that $\Phi^*P$ is an involutive, locally finitely generated, smooth singular subbundle of $TM$. Further, $\Phi$ being a symplectomorphism gives us that $\Phi^*P$ is generically Lagrangian in the sense of Definition \ref{poldef}.

\subsection{Symplectomorphisms and Sheaf Quantization}

For this section, fix a prequantum line bundle $(L,\nabla)\to (M,\omega)$, a polarization $P\subset TM$, and a symplectomorphism $\Phi:M\to M$. Recall that $\mfF_P$ denotes the sheaf of sections of $P$.

\begin{prop}\label{natiso1}
Let $F:L\to L$ be a lift of $\Phi$. Then the inverse of $F$ defines a natural isomorphism
\bge
\Phi_*(F^{-1})_*:\msS_{(P,\nabla)}\to \Phi_*\msS_{(\Phi^*P,F^*\nabla)},
\ene
where $\Phi_*(F^{-1})_*$ is the pushforward of the natural isomorphism $(F^{-1})_*$ by $\Phi$ as in Equation (\ref{pushnat}).
\end{prop}

\begin{proof}
Let $U\subset M$ be open. Recall that the lift $F$ defines an isomorphism
\bge
F_*:\Gamma_L(\Phi^{-1}(U))\to \Gamma_L(U)
\ene
from equation (\ref{nattran}). Fix $s\in \msS_{(P,\nabla)}(U)$. I will now show that $(F^{-1})_*s\in \msS_{(\Phi^*P,F^*\nabla)}(\Phi^{-1}(U))$, that is, for any $X\in \mfF_{\Phi^*P}(\Phi^{-1}(U))$, I want to show 
\bge
(F^*\nabla)_X F^{-1}_*s=0.
\ene
This is quite straightforward since by definition of $F^*\nabla$, we have
\bge
(F^*\nabla)_X F^{-1}_*s=F^{-1}_*(\nabla_{\Phi_*X} s).
\ene
By construction of $\Phi^*P$, we have $\Phi_*X\in \mfF_P(U)$. Hence, $\nabla_X s=0$ and thus $(F^*\nabla)_X F^{-1}_*s=0$.  Clearly $F_*:\Phi_*\msS_{(\Phi^*P,F^*\nabla)}\to \msS_{(P,\nabla)}$ is the inverse.
\end{proof}

\begin{cor}\label{quantiso}
For any lift $F:L\to L$ of $\Phi$, there exists a canonical isomorphism
\bge
Q_{shf}(M,P,\nabla)\to Q_{shf}(M,\Phi^*P,F^*\nabla).
\ene
\end{cor}

\begin{proof}
The natural isomorphisms of Proposition \ref{natiso1} give an isomorphism
\bge
\check{H}^\bullet(M,\msS_{(P,\nabla)})\to \check{H}^\bullet(M,\Phi_*\msS_{(\Phi^*P,F^*\nabla)})
\ene
All that's left to show is that there is an isomorphism between cohomology groups
\bge
\check{H}^\bullet(M,\Phi_*\msS_{(\Phi^*P,F^*\nabla)})\cong \check{H}^\bullet(M,\msS_{(\Phi^*P,F^*\nabla)}).
\ene

For sake of convenience, set $\msS=\msS_{(\Phi^*P,F^*\nabla)}$. For any open cover $\{U_j\}$ we have equality of chain complexes
\bge
\check{C}^\bullet(\{U_j\},\msS)=\check{C}^\bullet(\{\Phi^{-1}(U_j),\Phi_*\msS).
\ene
Thus, we get an isomorphism of cohomologies
\bge
\check{H}^\bullet(\{U_j\},\msS)\to\check{H}^\bullet(\{\Phi^{-1}(U_j)\},\Phi_*\msS).
\ene
It's easy to check that this isomorphism is compatible with refinement, hence induces a map
\bge
\check{H}^\bullet(M,\msS)\to \check{H}^\bullet(M,\Phi_*\msS).
\ene
It follows since $\Phi$ is a diffeomorphism that it is a bijection on open covers, hence the above map is in fact an isomorphism. See chapter 10 of Wedhorn\cite{torsten} for more details.
\end{proof}

\begin{proof}
(Of Main Theorem 1)
Applying Lemma \ref{KEYLEM} there exists a lift $F:L\to L$ of $\Phi$ such that $F^*\nabla=\nabla$. Corollary \ref{quantiso} then gives the desired isomorphism
\bge
Q_{shf}(M,P,\nabla)\to Q_{shf}(M,\Phi^*P,\nabla).
\ene
\end{proof}

\subsection{Symplectomorphisms and Bohr-Sommerfeld Leaves}
As before, fix a polarization $P\subset TM$ and a symplectomorphism $\Phi:M\to M$. 

\begin{lem}
If $\iota:B\to M$ is a leaf of $P$, then $\Phi^{-1}\circ \iota:B\to M$ is a leaf of $\Phi^*P$. 
\end{lem}

\begin{proof}
By construction of $\Phi^*P$ it is clear that if $\iota:B\to M$ is an immersed integral submanifold of $P$, then $\Phi^{-1}\circ \iota:B\to M$ is an immersed integral submanifold of $\Phi^*P$. Indeed, for any $x\in B$, by definition
\bge
\iota_*(T_xB)=P_{\iota(x)}.
\ene
Hence,
\bge
(\Phi^{-1}\circ\iota)_*(T_xB)=\Phi^{-1}_*P_{\iota(x)}=\Phi^*P_{\Phi^{-1}\circ\iota(x)}.
\ene

The only issue now is maximality of $\Phi^{-1}\circ \iota:B\to M$. Suppose $\iota':B'\to M$ is another integral submanifold of $\Phi^*P$ through $\Phi^{-1}\circ\iota(x)$ for some $x\in B$. Then, by the same argument, $\Phi\circ\iota':B'\to M$ is an integral submanifold of $P$ through $\iota(x)$. Thus, by the maximality of $\iota:B\to M$ there exists an open embedding
\bge
H:B'\to B
\ene
making the diagram commute
\bge
\begin{tikzcd}
& M &\\
B\arrow[ur,"\iota"] && B'\arrow[ul,"\Phi\circ\iota'"]\arrow[ll,"H"]
\end{tikzcd}
\ene
Since $\Phi$ is a diffeomorphism, it then follows that the below diagram commutes
\bge
\begin{tikzcd}
& M &\\
B\arrow[ur,"\Phi^{-1}\circ\iota"] && B'\arrow[ul,"\iota'"]\arrow[ll,"H"]
\end{tikzcd}
\ene
Hence, $\Phi^{-1}\circ\iota:B\to M$ is a maximal integral submanifold of $\Phi^*P$.
\end{proof}

For the sake of convenience, simply write $B$ for the data of a leaf $\iota:B\to M$ of $P$ and $\Phi^*B$ for the data of $\Phi^{-1}\circ\iota:B\to M$. It then becomes clear that $\Phi$ induces a bijection between the leaves of $P$ and $\Phi^*P$. 

\begin{cor}\label{leaflem}
If $Lf(P)$ denotes the set of leaves of $P$, then the map
\begin{align*}
Lf(P)\to Lf(\Phi^*P);\quad B\mapsto \Phi^*B.
\end{align*}
is a bijection. 
\end{cor}

\begin{lem}\label{natiso2}
Let $\Phi$ be a symplectomorphism and $F:L\to L$ a lift. Then, for any leaf $\iota:B\into M$, $F$ induces an isomorphism
\begin{equation}\label{bohriso}
(F_B)_*:H^0(B,\msS_{(\Phi^*P,F^*\nabla)}|_{\Phi^*B})\to H^0(B,\msS_{(P,\nabla)}|_B).
\end{equation}
\end{lem}

\begin{proof}
Fix a leaf $\iota:B\to M$ and a lift $F:L\to L$ of $\Phi$. Define an isomorphism of line bundles
\bge
\begin{tikzcd}
(\Phi^{-1}\circ \iota)^*L\arrow[rr,"F_B"]\arrow[dr]&&\iota^*L\arrow[dl]\\
&B&
\end{tikzcd}
\ene
by
\bge
F_B(x,z)=(x,F_{\iota(x)}(z)).
\ene
One then checks that the following identity holds
\bge
F_B^*(\iota^*\nabla)=(\Phi^{-1}\circ \iota)^*(F^*\nabla).
\ene
Thus, if $s:B\to (\Phi^{-1}\circ \iota)^*L$ is a section satisfying
\bge
(\Phi^{-1}\circ \iota)^*(F^*\nabla)_Xs=0
\ene
for all vector fields $X$ on $B$, then 
\bge
(\iota^*\nabla)_X (F_B)_*s=(F_B)_*((\Phi^{-1}\circ \iota)^*(F^*\nabla)_X s)=(F_B)_*(0)=0.
\ene
Hence, one obtains a map
\bge
(F_B)_*:H^0(B,\msS_{(\Phi^*P,F^*\nabla)}|_{\Phi^*B})\to H^0(B,\msS_{(P,\nabla)}|_B).
\ene
It's clear that $(F_B^{-1})_*$ is the inverse.
\end{proof}

\begin{proof}(Of Main Theorem 2)
By Lemma \ref{KEYLEM}, $\Phi$ admits a lift $F:L\to L$ such that $F^*\nabla=\nabla$. Due to Lemma \ref{natiso2}, the bijection in Corollary \ref{leaflem} restricts to a bijection
\bge
BS(P)\to BS(\Phi^*P).
\ene
Furthermore, taking the direct sum of the isomorphisms in equation (\ref{bohriso}), the desired isomorphism is given by
\bge
\bigoplus_{B\in BS(P)} (F_B)_*:Q_{BS}(M,\Phi^*P,\nabla)\to Q_{BS}(M,P,\nabla)
\ene
\end{proof}

\section{Application to Toric Geometry}
To finish off, let's discuss an application to toric varieties. There are easier proofs that don't require the machinery developed above, but it at least illustrates how the vanishing of the first cohomology with coefficients $\bC^\times$ can appear.

\begin{thm}
Let $(M,\omega)$ be a smooth compact symplectic toric variety together with a prequantum line bundle $(L,\nabla)\to M$. Then, for any choice of symplectomorphism $\Phi:M\to M$ and any choice of real polarization $P\subset TM$, there exists isomorphisms
\begin{align*}
Q_{shf}(M,P,\nabla)&\cong Q_{shf}(M,\Phi^*P,\nabla)\\
Q_{BS}(M,P,\nabla)&\cong Q_{BS}(M,\Phi^*P,\nabla)
\end{align*}
\end{thm}

\begin{proof}
Since $M$ is a compact toric variety, it follows that $\pi_1(M)=0$ \cite[First Proposition, Section 3.2]{fulton}.  Since 
\bge
H^1(M,\bC^\times)\cong \Hom(\pi_1(M),\bC^\times),
\ene
it then follows that $H^1(M,\bC^\times)=0$. Thus, the hypotheses of Main Theorems 1 and 2 hold. 
\end{proof}

As an application of this result, I obtained a kind of universality for quantization of toric manifolds under twisting. In more detail, let $T\cong (S^1)^n$ be an $n$-torus with $\mft=\text{Lie}(T)$ and $\mu:(M,\omega)\to \mft^*$ be a compact $T$-toric manifold. The momentum map $\mu$ naturally defines a polarization $P(\mu)\subset TM$ on $(M,\omega)$, where for each $x\in M$ we define
\begin{equation}\label{toricpol}
P(\mu)_x:=T_x\mu^{-1}(\mu(x)).
\end{equation}

Given a symplectomorphism $\Phi:M\to M$, one can twist the action of $T$ on $M$ by
\bge
T\times M\to M;\quad (t,x)\mapsto \Phi(t\cdot \Phi^{-1}(x)).
\ene
This is clearly a symplectic action. Furthermore, define 
\bge
\mu^{\Phi}:=\mu\circ \Phi^{-1}.
\ene
Then, $\mu^{\Phi}:(M,\omega)\to \mft^*$ is a $T$-toric manifold once again. Call this the twisting of $\mu:(M,\omega)\to \mft^*$ by $\Phi$.

Thus, each symplectomorphism $\Phi$ of a toric manifold generates a new polarization $P(\mu^\Phi)$ defined as in equation \ref{toricpol}, but with $\mu^{\Phi}$ replacing $\mu$. It is a triviality to unwind the definitions to show that
\bge
P(\mu^\Phi)=\Phi^*P(\mu).
\ene
This computation together with the previous theorem provides us with the following result.

\begin{cor}\label{toricinvariant}
Quantization of a compact toric manifold is invariant under twisting by symplectomorphisms.
\end{cor}

Of course this is nothing new as Hamilton in \cite[Theorem 8.10]{ham} showed for any compact toric manifold $\mu:(M,\omega)\to \mft^*$ that
\bge
Q_{shf}(M,P(\mu),\nabla)\cong \bigoplus_{\Lambda\cap \mu(M)^\circ}\bC,
\ene
where $\Lambda$ is the dual of the lattice $\ker(\exp:\mft\to T)$ and $\mu(M)^\circ$ is the interior of the associated Delzant polytope to $M$. Since the cardinality of $\Lambda\cap \mu(M)^\circ$ is a isomorphism invariant of a toric manifold, we immediately arrive at Corollary \ref{toricinvariant}.


\begin{thebibliography}{9}


\bibitem{jorgen}
Jørgen Ellegaard Andersen.
Geometric Quantization of Symplectic Manifolds with Respect to Reducible Non-Negative Polarizations.
\textit{Communications in Mathematical Physics}, vol 183, pp. 401-421, 1997.

\bibitem{holo} 
Iakovos Androulidakis and Georges Skandalis.
The Holonomy Groupoid of a Singular Foliation. 
\textit{Journal f\"{u}r die Reine und Angewandte Mathematik}, vol 2009, pp. 1-37, 2009.

\bibitem{bott} 
Raoul Bott and Loring W. Tu.
\textit{Differential Forms in Algebraic Topology}. 
Springer-Verlag, New York, N.Y., 1982.



\bibitem{fulton}
William Fulton.
\textit{Introduction to Toric Varieties}.
Princeton University Press, Princeton, N.J., 1993.


\bibitem{guil}
V. Guillemin and S. Sternberg.
The Gelfand-Cetlin System and Quantization of the Complex Flag Manifolds.
\textit{Journal of Functional Analysis}, vol. 52, 106-128, 1983.

\bibitem{ham}
Mark Hamilton.
Locally Toric Manifolds and Singular Bohr-Sommerfeld Leaves.
\textit{Memoirs of the American Mathematical Society}, vol. 207, 2010.


\bibitem{lisa}
Lisa C. Jeffrey and Jonathan Weitsman.
Bohr-Sommerfeld Orbits in the Moduli Space of Flat Connections and the Verlinde Dimension Formula.
\textit{Communications in Mathematical Physics}, vol. 150, pp. 593-630, 1992.

\bibitem{kostant}
Bertram Kostant.
\textit{Quantization and Unitary Representations}.
 Lectures in Modern Analysis and Applications III. Lecture Notes in Mathematics, vol 170. Springer, Berlin, Heidelberg, 1970.


\bibitem{coho}
Jedrzej Sniatycki.
\textit{On Cohomology Groups Appearing in Geometric Quantization}
Differential Geometrical Methods in Mathematical Physics, Lectures Notes in Mathematics, vol. 570, pp. 46-66, Springer, Berlin, 1977.


\bibitem{suss} 
H\'{e}ctor J. Sussmann.
Orbits of Families of Vector Fields and Integrability of Distributions. 
\textit{Transactions of the American Mathematical Society}, vol 180, pp. 171-188, 1973.




\bibitem{torsten}
Torsten Wedhorn.
\textit{Manifolds, Sheaves, and Cohomology}.
Springer Fachmedien Wiesbaden, 2016.

\bibitem{wood}
N. M. J. Woodhouse.
\textit{Geometric Quantization}.
Oxford University Press, Oxford, U.K., 1991.

\end{thebibliography}
\end{document}